\documentclass[12pt]{amsart}

\usepackage{amscd}
\usepackage{eucal}
\usepackage{amssymb}
\usepackage{epic}
\usepackage{graphics}
\usepackage{color}
\usepackage{mathrsfs}
\usepackage{multirow}
\usepackage{subfigure}
\usepackage{graphicx}

\numberwithin{equation}{section}
\numberwithin{table}{section}

\theoremstyle{plain}
\newtheorem{theorem}{Theorem}[section]

\theoremstyle{remark}

\newcommand{\BB}{{\mathcal{B}}}

\newcommand{\VV}{{\mathcal{V}}}

\newcommand{\TT}{\mathscr{T}}
\newcommand{\RRR}{{\mathbb{R}}}

\title[Integer Invariants]{Integer Invariants of an Incidence Matrix Related to Rota's Basis Conjecture}
\author[Bittner]{Stephanie Bittner}
\author[Ducey]{Joshua Ducey} 
\author[Guo]{Xuyi Guo}
\author[Oh]{Minah Oh}
\author[Zweber]{Adam Zweber}
\address{Dept.\ of Mathematics and Computer Science, Virginia Wesleyan College, Norfolk, VA 23502.} 
\email{snbittner@vwc.edu}
\address{Dept.\ of Mathematics and Statistics, James Madison University, Harrisonburg, VA 22807}
\email{duceyje@jmu.edu}
\address{Dept.\ of Mathematics, Stanford University, Stanford, CA 94305}
\email{xuyguo@stanford.edu}
\address{Dept.\ of Mathematics and Statistics, James Madison University, Harrisonburg, VA 22807}
\email{ohmx@jmu.edu}
\address{Mathematics Dept., Carleton College, Northfield, MN 55057}
\email{zwebera@carleton.edu}

\keywords{Rota's basis conjecture, incidence matrix, transversals, eigenvalues, invariant factors, Smith normal form}

\begin{document}
\begin{abstract}
We compute the spectrum and Smith normal form of the incidence matrix of disjoint transversals, a combinatorial object closely related to the $n$-dimensional case of Rota's basis conjecture.
\end{abstract}
\maketitle
\section{Introduction and Definitions}

Let $n$ be a positive integer, and consider a square array of $n^2$ distinct elements:
\begin{equation} \label{bases}
 \begin{aligned}
   a_1 && a_2 && \cdot\cdot\cdot && a_n \\
   b_1 && b_2 && \cdot\cdot\cdot && b_n \\
   \vdots && \vdots   && \vdots && \vdots \\         
   c_1 && c_2 && \cdot\cdot\cdot && c_n.
  \end{aligned}
\end{equation}
A transversal of the array $(\ref{bases})$ is an $n$-element set consisting of exactly one element from each row. Two transversals are said to be disjoint if and only if they are disjoint as sets. We write $\TT_n$ to denote the set of all transversals, ordered in some fashion.

These transversals together with the relation of disjointness will be the object of our study.  We can encode this information into an incidence matrix (more accurately, an adjacency matrix) as follows.
Define the \textit{incidence matrix of disjoint transversals} $A_n$ to be the $n^n\times n^n$ matrix whose rows and columns are indexed by $\TT_n$ such that
the $(i,j)$-entry of $A_n$ is one if the $i$-th transversal and the $j$-th transversal are disjoint and zero otherwise. 

The spectrum of $A_n$ does not depend on how the set of transversals was ordered, and so we may view it as an invariant of the incidence relation.  More fundamental is the \textit{Smith normal form} of $A_n$, which is unchanged even under independent row and column permutations of $A_n$ and hence describes the incidence relation at a more basic level. 

Recall that the Smith normal form of a (possibly nonsquare) integer matrix is a diagonal matrix of the same size, with the diagonal entries subject to certain divisibility conditions.  More formally, if $A$ is an $m \times n$ integer matrix, then there exist unimodular (invertible over the integers) matrices $P$ and $Q$ such that the matrix $PAQ = (d_{i,j})$ satisfies
\[
d_{i,j} = 0, \hbox{ for $i \neq j$}
\]
and
\[
d_{i,i} \hbox{ divides } d_{i+1,i+1}, \hbox{ for $1 \le i < \min\{m,n\}$}.
\]
This diagonal matrix is called the Smith normal form of $A$.  These diagonal entries $d_{i,i}$ are unique up to sign, and are called the invariant factors of the matrix $A$.

In the next section, we will compute the spectrum and the Smith normal form of $A_n$, the incidence matrix of disjoint transversals.  The matrices $A_n$ are actually the association matrices for the maximal distance in the Hamming association schemes $H(n,n)$, and their eigenvalues are known \cite{D:1973}.  We give another computation of the spectrum because it is relevant to how we will calculate the Smith normal form of $A_n$.  Our main result was conjectured in \cite{report:rbc} and first established in \cite{sin:snf} using character-theoretic methods.  Here we give an elementary proof.  In the final section, we explain the connection between the matrix $A_n$ and Rota's basis conjecture.

\section{The Spectrum and Smith Normal Form of $A_{n}$}

\begin{theorem}
Let $A_n$ be the incidence matrix of disjoint transversals. 
\begin{enumerate}
 \item \label{item1}
  The eigenvalues of $A_{n}$ are
\[
(-1)^{n-k} (n-1)^{k}
\]
occurring with multiplicity 
\[
\binom{n}{k} (n-1)^{n-k}
\]
for $0 \le k \le n$.

\item \label{item2}
The invariant factors of $A_{n}$ are
\[
(n-1)^{k}
\]
occurring with multiplicity
\[
\binom{n}{k} (n-1)^{n-k}
\]
for $0 \le k \le n$.
\end{enumerate}
\end{theorem}

\begin{proof}
Let $V=\RRR^{\TT_n}$, the vector space over the reals consisting of all formal linear combinations of elements of $\TT_n$. 
Then $A_{n}$ represents the linear map $\psi : V \to V$ defined by
\begin{equation} \label{A}
 \psi(\left\{a_{i_1}, b_{i_2}, \cdots , c_{i_n}\right\}) = \sum_{a_{j_1} \neq a_{i_1}, \cdots , c_{j_n} \neq c_{i_n}} \left\{a_{j_1}, b_{j_2}, \cdots , c_{j_n}\right\},
\end{equation}
for $\left\{a_{i_1}, b_{i_2}, \cdots , c_{i_n}\right\}\in\TT_n$.
That is, $\psi$ is the linear map defined by sending a transversal to the sum of the transversals disjoint to it.

Next, let $[n]$ denote the set $\left\{ 1, 2, \cdot\cdot\cdot, n \right\}$ and define
$W=\RRR^{[n]}$, the vector space over the reals consisting of all formal linear combinations of elements of $[n]$.
If $B_n$ is the $n\times n$ matrix whose diagonal entries are zero and all other entries are one,
then $B_n$ represents the linear map $\theta : W \to W$ defined by
\begin{equation} \label{B}
\theta(i) = \sum_{j \in [n],  j \neq i} j,
\end{equation}
for $i\in [n]$.
That is, $\theta$ sends an element in $[n]$ to the (formal) sum of all the elements of $[n]$ distinct from it.
It is easy to see that $n-1$ is an eigenvalue of $B_n$ with multiplicity one, 
and $-1$ is an eigenvalue with multiplicity $n-1$ (eigenvectors will have entries that sum to zero).  

Now, in order to see the connection between (\ref{A}) and (\ref{B}), let us 
take the tensor product of $W$ with itself $n$ times and consider the induced linear map
$\theta^{\otimes n} : W^{\otimes n} \to W^{\otimes n}$ defined by 
\begin{align*}
\theta^{\otimes n}(i_{1} \otimes \cdots \otimes i_{n}) &=  \theta(i_{1}) \otimes \cdots \otimes \theta(i_{n}).
\end{align*}
Since the tensor product respects the distributive law over vector addition, we have that
\begin{equation} \label{A_B}
\begin{aligned}
\theta^{\otimes n}(i_{1} \otimes \cdots \otimes i_{n}) & =  \theta(i_{1}) \otimes \cdots \otimes \theta(i_{n}), \\
& =  (\sum_{j_{1} \neq i_{1}}j_{1}) \otimes \cdots \otimes (\sum_{j_{n} \neq i_{n}}j_{n}), \\
& =  \sum_{j_{1} \neq i_{1}, \cdots, j_{n} \neq i_{n}} j_{1} \otimes \cdots \otimes j_{n}.
\end{aligned}
\end{equation}
It is easy to see that the vector spaces $V$ and $W^{\otimes n}$ are isomorphic under the map that sends a transversal $\left\{ a_{i_1}, b_{i_2}, \cdots , c_{i_{n}} \right\}$ to the simple tensor $i_1 \otimes i_2 \otimes \cdots \otimes i_{n}$, and after identifying these spaces the maps (\ref{A}) and (\ref{A_B}) are the same.  
In other words, $A_n$ can be obtained by taking the Kronecker product of $B_n$ with itself $n$ times.
Therefore, part (\ref{item1}) of the theorem is proved if we show that the eigenvalues of $\theta^{\otimes n}$ are 
\[
(-1)^{n-k} (n-1)^{k}
\]
occurring with multiplicity 
\[
\binom{n}{k} (n-1)^{n-k}
\]
for $0 \le k \le n$.
To see this, take a basis of eigenvectors $\{x_{i}\}_{i=1}^{n}$ for the map $\theta : W \to W$ in (\ref{B}).  Let $x_1$ have eigenvalue $n-1$ and let $x_2, x_3, \cdots , x_n$ each have eigenvalue $-1$.  Let $x_{i_1} \otimes \cdots \otimes x_{i_{n}}$ be an arbitrary basis vector for $W^{\otimes n}$.  Then it is easy to see that this basis vector is in fact an eigenvector for $\theta^{\otimes n}$ and both the eigenvalue and multiplicity depend on how many times $x_{1}$ occurs in this simple tensor.  Explicity, if $x_{1}$ occurs exactly $k$ times in the simple tensor $x_{i_1}\otimes \cdots \otimes x_{i_{n}}$ then 
\begin{eqnarray*}
\theta^{\otimes n}(x_{i_1} \otimes \cdots \otimes x_{i_{n}}) &=& \theta(x_{i_1}) \otimes \cdots \otimes \theta(x_{i_{n}}) \\
& = & (-1)^{n-k} (n-1)^{k} x_{i_1}\otimes \cdots \otimes x_{i_{n}},
\end{eqnarray*}
so $(-1)^{n-k} (n-1)^{k}$ is an eigenvalue of $\theta^{\otimes n}$.  Additionally, the number of basis vectors containing a factor of $x_{1}$ exactly $k$ times is $\binom{n}{k} (n-1)^{n-k}$. This completes the proof of part (\ref{item1}). 

Next, let us prove part (\ref{item2}) of the theorem.  It is amazing that a nearly identical argument goes through to give the invariant factors of $A_{n}$.  Set $V = \mathbb{Z}^{\mathcal{T}_{n}}$ and $W = \mathbb{Z}^{[n]}$, the free modules over the integers consisting of all formal linear combinations of elements of $\TT_n$ and $[n]$ respectively. The maps (\ref{A}) and (\ref{B}) should now be viewed as homomorphisms of free abelian groups. It is an easy exercise using integral row and column operations to show that $B_n$ has Smith normal form:
\[ S=
\begin{bmatrix}
1 & & & & \\
& 1 & & & \\
& & \ddots & & \\
& & & 1 & \\
& & & & n-1 
\end{bmatrix}.
\]
Therefore, $1$ is an invariant factor with multiplicity $n-1$ and $n-1$ is an invariant factor with multiplicity $1$. This means that we can find bases $\{x_{i}\}_{i=1}^{n}$ and $\{y_{i}\}_{i=1}^{n}$ of $W$ so that 
\begin{align*}
\theta(x_1) &= (n-1)y_1 \\
\theta(x_{i}) &= y_{i}, 
\end{align*}
for $2 \le i \le n$.  The matrix $A_{n}$ still represents the homomorphism $\theta^{\otimes n} : W^{\otimes n} \to W^{\otimes n}$, and the sets $\{x_{i_1} \otimes \cdots \otimes x_{i_{n}}\}$ and $\{y_{i_1} \otimes \cdots \otimes y_{i_{n}}\}$ are bases of $W^{\otimes n}$.  If $x_1$ occurs exactly $k$ times in the basis vector $x_{i_1} \otimes \cdots \otimes x_{i_{n}}$, then we have
\begin{align*}
\theta^{\otimes n}(x_{i_1} \otimes \cdots \otimes x_{i_{n}}) &= \theta(x_{i_1}) \otimes \cdots \otimes \theta(x_{i_{n}}), \\
& =  (n-1)^{k} y_{i_1}\otimes \cdots \otimes y_{i_{n}}.
\end{align*}
Hence, the invariant factors of $A_{n}$ are
\[
(n-1)^{k}
\]
occurring with multiplicity
\[
\binom{n}{k} (n-1)^{n-k}
\]
for $0 \le k \le n$, and this completes the proof.
\end{proof}

\section{Relationship to Rota's Basis Conjecture}
The following conjecture was made by Gian-Carlo Rota in 1989 \cite{HR:1994}.

Suppose one has $n$ bases $\BB_{1}, \BB_{2}, \cdots , \BB_{n}$ of an $n$-dimensional vector space $\VV$.  Form an $n\times n$ array, with the elements of $\BB_{i}$ forming the $i$-th row.  Then Rota's basis conjecture states that there is a way to independently permute the entries of each row so that all the columns of the array are also bases of $\VV$.

This conjecture is very general, stated for any finite dimensional vector space over any field.  The basis conjecture is implied for even dimensions and fields of specific characteristics (in particular, characteristic zero) by the Alon-Tarsi conjecture concerning even and odd Latin squares \cite{AT:1992, Onn:1997}.  Indeed, much of the recent progress on Rota's basis conjecture has been through investigations of Latin squares \cite{Drisko:1997, Glynn:2010}.  The conjecture also generalizes to an identical statement about bases in a rank $n$ matroid that has been fully settled only for $n \leq 3$ \cite{Chan:1995}.

If we assign the elements of the $i$-th basis $\BB_{i}$ to the positions in the $i$-th row of the array (\ref{bases}), then Rota's basis conjecture asserts that there will exist $n$ mutually disjoint transversals where each transversal corresponds to a basis of $\VV$.  Now, which transversals will correspond to bases will depend on the original choice of $\BB_{1}, \BB_{2}, \cdots , \BB_{n}$, and each such choice will distinguish a principal submatrix of $A_{n}$.  Thus in this sense the matrix $A_{n}$ contains the information needed to verify any particular instance of Rota's basis conjecture.  It is hoped that the matrix $A_{n}$ may offer a new approach this problem.

\section{Acknowledgements}
The authors would like to thank Peter Sin for helpful discussions.  The authors also acknowledge the helpful computations produced by Michael Cheung.  Discussion of this work originated from an REU (Research Experience for Undergraduates) project at James Madison University in June 2012.  This work was supported in part by the NSF under DMS-1004516.

\bibliographystyle{siam}

\end{document}